\documentclass[a4paper,11pt]{amsart}
\usepackage{geometry}
\usepackage{fullpage}
\usepackage[parfill]{parskip}   % Activate to begin paragraphs with an empty line rather than an indent
\usepackage{graphicx}
\usepackage{amssymb}
\usepackage{amsmath}
\usepackage{epstopdf}
\usepackage{amsmath}

\usepackage{amsmath,amsthm,amscd,amssymb}
\usepackage{latexsym}
\usepackage[colorlinks,citecolor=red,pagebackref,hypertexnames=false]{hyperref}
 
\numberwithin{equation}{section}

\theoremstyle{plain}
\newtheorem{theorem}{Theorem}[section]

\newtheorem{lemma}[theorem]{Lemma}

\newtheorem{conjecture}[theorem]{Conjecture}

\theoremstyle{definition}

\theoremstyle{remark}
\newtheorem{remark}[theorem]{Remark}

\newtheorem{case[theorem]}{Case}

\def \R{{\mathbb R}}

\def \C{{\mathbb C}}

\def\norm#1.#2.{\lVert#1\rVert_{#2}}

\def\R{\mathbb R}

\def \S{{\mathcal S}}

\title{Uncertainty Principles for the Fourier and the Short-Time Fourier Transforms}
 
\author{Anirudha Poria}

\address{Department of Mathematics,
Indian Institute of Science,
Bengaluru 560012, Karnataka, India.}

\email{anirudhap@iisc.ac.in}
\keywords{Uncertainty principles; short-time Fourier transform; Beurling's theorem.}
\subjclass[2010]{Primary 42B10; Secondary 94A12.}

\date{\today}

\begin{document}
\maketitle
\begin{abstract} 
The aim of this paper is to establish a few uncertainty principles for the Fourier and the short-time Fourier transforms. Also, we discuss an analogue of Donoho--Stark uncertainty principle and provide some estimates for the size of the essential support of the short-time Fourier transform.   
\end{abstract}    

\section{Introduction and statement of the results} 

The uncertainty principle states that a non-zero function and its Fourier transform cannot be simultaneously sharply localized. We consider the Fourier transform on $\R^d$ to be normalized as $$\hat{f}(\xi)=\int_{\R^d} f(x) e^{- 2 \pi i x \cdot \xi} \; dx,$$ where $x \cdot \xi$ is the standard inner product on $\R^d$. There are various forms of the uncertainty principle. The most remarkable result is due to Beurling, which states that:
\begin{theorem}[Beurling]\label{Beurling}
Let $f \in L^2(\mathbb{R}^d)$ be such that
\[ \iint_{\mathbb{R}^{2d}} |f(x) \hat{f}(\xi)| e^{2  \pi |x \cdot \xi| }\;dxd\xi < \infty.\]
Then $f = 0$ almost everywhere.
\end{theorem}
Its proof was published much later in 1991 by H\"ormander \cite{Hor91}. We can obtain the well-known uncertainty principles of Hardy, Cowling--Price and Gelfand--Shilov as corollaries to Theorem \ref{Beurling} (see \cite{Thn04}). For the purpose of this paper, we state here Hardy's and Cowling--Price's theorems. In 1933, Hardy \cite{Har33} proved the following uncertainty principle:
\begin{theorem}[Hardy]\label{Hardy}
Let $f \in L^2(\R^d)$, and assume that
\[ |f(x)| \leq C e^{-a \pi x^2}  \qquad \mathrm{and} \qquad  |\hat{f}(\xi)| \leq C e^{-b \pi \xi^2} \]
for some constants $a, b, C > 0$. Then three cases can occur.
\begin{enumerate}
\item[$(i)$] If $ab=1$, then $f(x)=C e^{-a \pi x^2}.$
\item[$(ii)$] If $ab> 1$, then $f \equiv 0$.
\item[$(iii)$] If $ab<1$, then any finite linear combination of Hermite functions satisfies these decay conditions.
\end{enumerate} 
\end{theorem}
In 1983, Cowling and Price \cite{Cow83} generalized this theorem by replacing point wise Gaussian bounds for $f$
by Gaussian bounds in $L^p$ sense and in $L^q$ sense for $\hat{f}$ as well. More precisely, they proved the
following theorem:
\begin{theorem}[Cowling--Price]\label{Cow-Pri}
Let $f:\R^d \to \C$ be a measurable function such that
\begin{enumerate}
\item[$(i)$] $\Vert e^{a \pi x^2} f  \Vert_p < \infty,$
\item[$(ii)$] $\Vert e^{b \pi \xi^2} \hat{f}  \Vert_q < \infty,$
\end{enumerate}
where $a, b>0$ and $1 \leq p, q \leq \infty$ such that $\min(p, q)$ is finite. If $ab \geq 1$, then $f=0$ almost everywhere. If $ab < 1$, then there exist infinitely many linearly independent functions satisfying $(i)$ and $(ii)$.
\end{theorem}
The Beurling's theorem was further generalized in 2003
by Bonami, Demange and Jaming \cite{Bon03} as follows:
\begin{theorem}\label{Bonami}
Let $f \in L^2(\mathbb{R}^d)$ be such that
\begin{equation}\label{Bon-eq}
\iint_{\mathbb{R}^{2d}} \dfrac{|f(x) \hat{f}(\xi)|}{(1 + \vert x \vert + \vert \xi \vert)^N} \; e^{2  \pi |x \cdot \xi| }\;dxd\xi < \infty
\end{equation}
for some $N \geq 0$. Then $f=0$ almost everywhere whenever $N \leq d$.  If $N >d$, then $f(x)= P(x)e^{-a \pi x^2}$ where $P$ is a polynomial of degree $< \frac{N-d}{2}$ and $a>0$.
\end{theorem}
Over the years, analogues of Beurling's theorem have been extended to different settings (see \cite{Thn04}). For a more detailed study of uncertainty principles, we refer to the book of Havin and J\"oricke \cite{hav94}. In time-frequency analysis, another tool of investigation is the short-time Fourier transform (STFT). We first write the definition of the STFT. 

Let $g \in \S(\R^d)$ be a fixed window function. Then the STFT of $f \in \S'(\R^d)$ with respect to $g$ is defined to be the function on $\R^d \times \hat{\R}^d$ given by
\[ V_g f (x,\xi)=\int_{\R^d} f(t) \overline{g(t-x)} e^{-2 \pi i \xi \cdot t} dt. \]
Gr\"ochenig and Zimmermann \cite{Gro01} have shown that it is possible to derive new uncertainty principles for the STFT from uncertainty principles for the pair $(f, \hat{f})$ using a fundamental identity for the STFT. They proved a version of Hardy's theorem for the STFT:
\begin{theorem}\label{GroZim}
Let $(g,f) \in \S \times \S'(\R^d)$, and assume that
$ |V_g f (x,\xi)| \leq C e^{- \pi (x^2+\xi^2)/2},  $
and that $V_g f$ does not vanish identically. Then
$V_g f (x,\xi)=C \; e^{2 \pi i ( \zeta_{0} \cdot x - \xi \cdot z_{0})} e^{- \pi (x^2+\xi^2)/2}  e^{- \pi i \xi \cdot x} $
for some $(z_0 , \zeta_0) \in \R^d \times \hat{\R}^d$, and $f$ and $g$ are multiples of $ e^{2 \pi i  \zeta_{0} \cdot t} \; e^{- \pi (t-z_0)^2}$.
\end{theorem}
Considerable attention has been paid to prove an analogue of Beurling's theorem for the STFT (see \cite{Bon03, Dem05, Gro2001, Gro03}). Bonami, Demange and Jaming \cite{Bon03} proved the following version of Beurling's theorem for the STFT. They used the $L^2$-norm instead of the $L^1$-norm of the STFT.
\begin{theorem}
Let $f,g \in L^2(\R^d)$ be non identically vanishing. If
\[ \iint_{\mathbb{R}^{2d}} \dfrac{|V_g f(x,\xi)|^2}{(1 + \vert x \vert + \vert \xi \vert)^N}  \; e^{\pi (x^2+\xi^2)} \;dxd\xi < \infty, \]
then there exists $a, w \in \R^d$ such that both $f$ and $g$ are of the form $P(x) e^{2 \pi i  w \cdot x} e^{- \pi (x-a)^2}$, where $P$ is a polynomial.
\end{theorem}
Demange \cite{Dem05} improved the above theorem to the following sharper version of Beurling's theorem for the STFT:
\begin{theorem}\label{Demange}
Let $f,g \in L^2(\R^d)$. If there exists an $N \geq 0$ such that 
\begin{equation}\label{Dem-eq}
\iint_{\mathbb{R}^{2d}} \dfrac{|V_g f(x,\xi)|}{(1 + \vert x \vert + \vert \xi \vert)^N}  \; e^{\pi |x \cdot \xi| } \;dxd\xi < \infty,
\end{equation}
then either $f$ or $g$ is identically zero, or both can be written as
\[ f(x)= P(x) e^{-a x^2 - 2 \pi i  w \cdot x} \quad \mathrm{and}  \quad g(x)= Q(x) e^{-a x^2 - 2 \pi i  w \cdot x}, \]
with $P$ and $Q$ polynomials whose degrees satisfy $\deg(P ) + \deg(Q)<N-d, \; w \in \C^d$ and $a>0$. The converse is also true. In particular, for $N \leq d$, $f$ or $g$ are identically vanishing.
\end{theorem}
It has been conjectured that a result similar to Beurling's theorem is also true for the STFT. Gr\"ochenig \cite{Gro03} posed the following conjecture as a version of Beurling's theorem for the STFT:
\begin{conjecture}\label{Con}
Assume that $f,g \in L^2(\R^d)$. If \[ \iint_{\mathbb{R}^{2d}} |V_g f(x,\xi)| \; e^{\pi |x \cdot \xi|} \;dxd\xi < \infty, \]
then $f \equiv 0 $ or $g \equiv 0 $. 
\end{conjecture}
But so far upto our knowledge, this result has not been proved. However, Gr\"ochenig \cite{Gro03} proved a weaker version of this conjecture as follows:
\begin{theorem}\label{Gro1}
Assume that $f,g \in L^2(\R^d)$. If
\[ \iint_{\mathbb{R}^{2d}} |V_g f(x,\xi)| \; e^{\pi (x^2+\xi^2)/2} \;dxd\xi < \infty, \]
then $f \equiv 0 $ or $g \equiv 0 $.
\end{theorem}
Further, Gr\"ochenig obtained an estimate for the size of the essential support of $V_g f$, analogous to the uncertainty principle of Donoho and Stark \cite{Don89} for the pair $(f, \hat{f})$. Several interesting versions of uncertainty principle have been studied by various authors for the STFT. We refer the reader to \cite{Jam98, Lie90, Mal10} and the references therein. 

The aim of this paper is to prove Conjecture \ref{Con} and
a few uncertainty principles for the Fourier and the STFT. We investigate the following problems: 

If $\| e^{\pi(x^2+\xi^2)/2} \; V_g f\|_{L^p(\R^{2d})} < \infty$, or  $\| e^{\pi |x \cdot \xi|} \; V_g f\|_{L^p(\R^{2d})} < \infty$, or $\| e^{2 \pi |x \cdot \xi|} \; f\; \hat{f} \; \|_{L^p(\R^{2d})} < \infty$, then what we can say about the functions $f$ and $g$.

More precisely, we establish the following problems:
\begin{theorem}\label{Th1}
Let $1 \leq p < \infty$ and $f,g \in L^2(\R^d)$ be non identically vanishing. If
\begin{equation}\label{eq01}
\iint_{\mathbb{R}^{2d}} |V_g f(x,\xi)|^p \; e^{\pi p (x^2+\xi^2)/2} \;dxd\xi < \infty,
\end{equation}
then $f \equiv 0 $ or $g \equiv 0 $.
\end{theorem}

\begin{theorem}\label{Th2}
Let $1 \leq p < \infty$ and $f,g \in L^2(\R^d)$ be non identically vanishing. If 
\begin{equation}\label{eq02}
\iint_{\mathbb{R}^{2d}} |V_g f(x,\xi)|^p \; e^{\pi p |x \cdot \xi|} \;dxd\xi < \infty,
\end{equation}
then $f \equiv 0 $ or $g \equiv 0 $. 
\end{theorem}

\begin{theorem}\label{Th3}
Let $1 \leq p < \infty$ and $f \in L^2(\R^d)$ be such that
\begin{equation}\label{eq03}
\iint_{\mathbb{R}^{2d}} |f(x) \hat{f}(\xi)|^p \; e^{2  \pi p |x \cdot \xi| }\;dxd\xi < \infty.
\end{equation}
Then $f \equiv 0$.
\end{theorem}
Finally, we extend the results of Gr\"ochenig \cite{Gro03} and prove the following versions of the uncertainty principle about an estimate on the size of the essential support of the STFT:
\begin{theorem}\label{Th4}
Let $f$ and $g \in L^2(\R^d)$. If $U \subseteq \R^{2d}$ and $\epsilon \geq 0$ are such that
\[\iint_U |V_g f(x,\xi)| \;dx d\xi \geq (1-\epsilon) \|f \|_2 \|g \|_2, \]
then 
\[ |U| \geq  (1-\epsilon)^{\frac{p}{p-1}} \left( \frac{p}{2}  \right)^{\frac{d}{p-1}} \qquad \text{for all  } \; p \geq 2.  \]
\end{theorem}

\begin{theorem}\label{Th5}
Suppose that $f, \; g \in L^2(\R^d)$, $U \subseteq \R^{2d}$ and $\epsilon \geq 0$ are such that 
\begin{equation}\label{eq003}
\iint_U |V_g f(x,\xi)|^p dx d\xi \geq (1-\epsilon) \|V_g f \|^p_1.
\end{equation}
Then 
\[ |U| \geq 2^{\frac{2pd}{2-p}}  (1-\epsilon)^{\frac{2}{2-p}} \qquad \text{for all  } \; 1 \leq p < 2. \]
\end{theorem}
The paper is organized as follows. In Section \ref{sec1}, we recall some of the properties of the STFT. Then, in Section \ref{sec2}, we prove the main results and discuss some consequences.

\section{The short-time Fourier transform}\label{sec1}

Translation and modulation are defined by $T_x f(t)=f(t-x)$ and  $M_\xi f(t)= e^{2 \pi i t \cdot \xi} f(t)$, where $t, x, \xi \in \R^d$. Using this notation, the STFT can be written as 
\[V_g f(x, \xi)=\langle f, M_\xi T_x g \rangle = \widehat{(f \cdot T_x \bar{g})} (\xi). \]

For a detailed discussion of STFT see \cite{Gro2001}. To prove uncertainty principles for the STFT we need to construct an expression derived from $V_g f$ that is invariant under the $2d$-dimensional Fourier transform. This kind of function was obtained by Jaming in \cite{Jam98}, which played a central role in obtaining certain uncertainty theorems for the STFT. We recall the following identities for the STFT from \cite{Gro03}, which we need for the proof of Theorem \ref{Th1}.

\begin{lemma}\label{lem1}
Assume that $f_1, f_2, g_1, g_2 \in L^2(\R^d)$. Then
\begin{equation}\label{eq001}
\widehat{\left(V_{g_1}f_1 \overline{V_{g_2}f_2}\right)}(x,\xi)=\left(V_{f_2}f_1 \overline{V_{g_2}g_1}\right) (-\xi, x) .
\end{equation}
\end{lemma}

Putting $f_1= f_2$, $g_1= g_2$, and $x=\xi=0$ in (\ref{eq001}), we obtain the isometry property of the STFT:
\begin{equation}\label{eq002}
\| V_g f \|^2_{L^2(\R^{2d})}=\|f \|^2_2 \; \|g\|^2_2.
\end{equation}

\begin{lemma}\label{lem2}
\begin{enumerate}
\item[$(i)$] For $f,g \in L^2(\R^d)$, the function
\[ F(x,\xi)=e^{ 2 \pi i x \cdot \xi} \;V_g f(x,\xi) \; V_g f(-x, -\xi)  \]
satisfies
\[\hat{F}(x,\xi)=F(-\xi, x).\]
\item[$(ii)$] Consider the family of functions defined as 
\[ F_{(z, \zeta)}(x,\xi)=e^{ 2 \pi i x \cdot \xi} \;V_g (M_\zeta T_z f)(x,\xi) \; V_g (M_\zeta T_z f)(-x,-\xi).   \] 
Then
\[ \widehat{F_{(z, \zeta)}}(x,\xi)=F_{(z, \zeta)}(-\xi, x) \qquad for \;all \; (z, \zeta) \in \R^{2d}. \]
\end{enumerate}
\end{lemma}

The above lemmas contain a fundamental identity for the STFT. They have been derived and used to prove certain uncertainty principles for the STFT (see \cite{Jam98, Gro01, Gro03}). The advantage of the identity is that the auxiliary function $F_{(z, \zeta)}$ inherits many properties from $V_g f$. For instance, if $V_g f$ possesses a certain decay, then  $F_{(z, \zeta)}$ has a similar decay.

\section{Proofs of main results}\label{sec2}

A simple consequence of the Cowling--Price's theorem is obtained in the following lemma. 

\begin{lemma}\label{lem3}
Let $1 \leq p < \infty$ and $f \in L^2(\R^d)$ be such that
\begin{equation}\label{eq1}
\iint_{\mathbb{R}^{2d}} |f(x) \hat{f}(\xi)|^p \; e^{\pi p (x^2+\xi^2)} \;dxd\xi < \infty.
\end{equation}
Then $f \equiv 0$.
\end{lemma}
\begin{proof}
Let $e_\pi (x)=e^{\pi x^2}$ for $x \in \R^d$. If $f \in L^2(\R^d)$ satisfies (\ref{eq1}), then we have
\[ \Vert e_\pi f \Vert_p^p \; \Vert e_\pi \hat{f} \Vert_p^p= \iint_{\mathbb{R}^{2d}} |f(x) \hat{f}(\xi)|^p \; e^{\pi p (x^2+\xi^2)} \;dxd\xi < \infty. \] 
Thus, the assumption (\ref{eq1}) implies that $e_\pi f$ and $e_\pi \hat{f}$ are both in $L^p(\R^d)$ for $1 \leq p < \infty$. Hence, the conditions of Cowling--Price's theorem (Theorem \ref{Cow-Pri}) are satisfied for $f$ and we conclude that $f \equiv 0$.
\end{proof}

\begin{proof}[Proof of Theorem \ref{Th1}]
To simplify the notation, we use $X=(x,\xi)$ and $Z=(z, \zeta) \in \R^{2d}$. We write $UX =(-\xi, x) $ for the rotation and $X^2=x^2+\xi^2$. We consider the family of functions defined in Lemma \ref{lem2} as
\[ F_{(z, \zeta)}(x,\xi)= F_Z (X) =e^{ 2 \pi i x \cdot \xi} \;V_g (M_\zeta T_z f)(x,\xi) \; V_g (M_\zeta T_z f)(-x,-\xi).   \] 
By Lemma \ref{lem2} $(ii)$ we have $ \widehat{F_Z}(\Omega)=F_Z(U \Omega)$. To apply Lemma \ref{lem3}, we need to show that
\begin{equation}\label{eq2}
\int_{\mathbb{R}^{2d}} \int_{\mathbb{R}^{2d}} |F_Z(X) \widehat{F_Z}(\Omega)|^p \; e^{\pi p (X^2+\Omega^2)} \;dX d\Omega < \infty.
\end{equation}
Thus it suffices to show that 
\begin{eqnarray}\label{eq3}
&& \int_{\mathbb{R}^{2d}} \int_{\mathbb{R}^{2d}} |F_Z(X)|^p \;|F_Z (U \Omega)|^p \; e^{\pi p (X^2+\Omega^2)} \;dX d\Omega \nonumber \\ 
&& = \left( \int_{\mathbb{R}^{2d}} |F_Z(X)|^p \;e^{\pi p X^2} dX  \right)^2 :=B(Z)^2 < \infty.
\end{eqnarray}
Since
\[ |V_g (M_\zeta T_z f)(x,\xi)|=|V_g f (x-z, \xi - \zeta)|=|V_g f (X-Z)|  \]
and 
\[ \frac{1}{2}(X-Z)^2+\frac{1}{2}(-X-Z)^2=X^2+Z^2,  \]
the expression for $B(Z)$ can be written as
\[B(Z)=e^{-\pi p Z^2} \int_{\mathbb{R}^{2d}} |V_g f (X-Z)|^p e^{\pi p (X-Z)^2/2} \; |V_g f (-X-Z)|^p e^{\pi p (-X-Z)^2/2} \; dX.  \]
Let $\Psi(X)=|V_g f(X)|^p e^{\pi p X^2/2}$, then assumption (\ref{eq01}) implies that 
\[\int_{\mathbb{R}^{2d}} \Psi(X) \; dX=\int_{\mathbb{R}^{d}} \int_{\mathbb{R}^{d}} |V_g f(x, \xi)|^p e^{\pi p (x^2+\xi^2)/2} \; dxd\xi<\infty , \] 
and so $\Psi \in L^1(\mathbb{R}^{2d})$. Moreover
\begin{eqnarray*}
B(Z) &=& e^{-\pi p Z^2} \int_{\mathbb{R}^{2d}} \Psi(X-Z) \; \Psi(-X-Z) \; dX \\
& = & e^{-\pi p Z^2} \int_{\mathbb{R}^{2d}} \Psi(X) \; \Psi(-2Z-X) \; dX \\
& = & e^{-\pi p Z^2} (\Psi \ast \Psi)(-2Z).
\end{eqnarray*}
Since $\Psi \in L^1(\mathbb{R}^{2d})$, we have $\Psi \ast \Psi  \in L^1(\mathbb{R}^{2d})$ and hence $\Psi \ast \Psi(-2Z)< \infty $. Thus $B(Z) < \infty $ for almost all $Z \in \mathbb{R}^{2d}$. Thus the condition (\ref{eq1}) of Lemma \ref{lem3} is satisfied for $F_Z(X)$ and we conclude that
\begin{equation}\label{eq4} 
|F_Z(X)|=|V_g f(X-Z) \; V_g f(-X-Z)|=0 
\end{equation}
for almost all $Z \in \mathbb{R}^{2d}$. Since $F_Z(X)$ is jointly continuous in $X$ and $Z$, (\ref{eq4}) is true for all $X, Z \in \mathbb{R}^{2d}$. Consequently
\[ |F_Z(0)|=|V_g f(-Z)|^2=0 \qquad \mathrm{for\; all\;} Z \in \mathbb{R}^{2d}. \]
This implies that either $f \equiv 0$ or $g \equiv 0$.
\end{proof}

\begin{proof}[Proof of Theorem \ref{Th2}]
Let $1 \leq p,q <\infty$ with $1/p+1/q=1$. We choose $N>0$, such that 
\begin{equation}\label{eq5}
\iint_{\mathbb{R}^{2d}} \frac{1}{(1+|x|+|\xi|)^{Nq}} \; dx d\xi < \infty.
\end{equation}
Then using H\"older's inequality, we get
\begin{eqnarray}\label{eq6}
&& \iint_{\mathbb{R}^{2d}} \frac{|V_g f (x, \xi)|}{(1+|x|+|\xi|)^{N}} \; e^{\pi |x \cdot \xi|} \; dx d\xi  \nonumber \\
&& \leq  \left( \iint_{\mathbb{R}^{2d}} |V_g f (x, \xi)|^p \; e^{\pi p |x \cdot \xi|} \; dx d\xi \right)^{\frac{1}{p}} \left( \iint_{\mathbb{R}^{2d}} \frac{1}{(1+|x|+|\xi|)^{Nq}} \; dx d\xi \right)^{\frac{1}{q}}. 
\end{eqnarray}
Thus the assumptions (\ref{eq02}) and (\ref{eq5}) imply that 
\[\iint_{\mathbb{R}^{2d}} \frac{|V_g f (x, \xi)|}{(1+|x|+|\xi|)^{N}} \; e^{\pi |x \cdot \xi|} \; dx d\xi < \infty.  \]
Hence, the condition (\ref{Dem-eq}) of Theorem \ref{Demange} is satisfied for $V_g f$ and we conclude that either $f$ or $g$ is identically zero, or both can be written as
\begin{equation}\label{eq7}
f(x)= P(x) e^{-a x^2 - 2 \pi i  w \cdot x} \quad \mathrm{and}  \quad g(x)= Q(x) e^{-a x^2 - 2 \pi i  w \cdot x},
\end{equation}
with $P$ and $Q$ polynomials whose degrees satisfy $\deg(P ) + \deg(Q)<N-d, \; w \in \C^d$ and $a>0$. Indeed, we show that if $f$ and $g$ are as in (\ref{eq7}) and $V_g f$ satisfies (\ref{eq02}), then $f \equiv 0$ or $g \equiv 0$.

Let $f$ and $g$ are of the form given in (\ref{eq7}), then
\begin{equation*}
V_g f(x, \xi)=R(x,\xi) e^{- \pi i x \cdot \xi} e^{-(\pi/2)a^{-1} (\xi+2 i w_2)^2} e^{-(\pi/2)a x^2} e^{-2 \pi i x \cdot w_1},
\end{equation*}
where $R$ is a polynomial of degree $\deg(P) + \deg(Q)$ and $w=w_1+ i w_2$. Therefore, 
\[|V_g f(x, \xi)|=|R(x,\xi)| e^{-(\pi/2)a^{-1} (\xi^2-4 w^2_2)} e^{-(\pi/2)a x^2}. \]
Since $V_g f$ satisfies (\ref{eq02}), we have 
\begin{eqnarray*}
&& \iint_{\mathbb{R}^{2d}} |V_g f (x, \xi)|^p \; e^{\pi p |x \cdot \xi|} \; dx d\xi \\
&& = e^{2 \pi a^{-1} p w^2_2 } \iint_{\mathbb{R}^{2d}} |R(x,\xi)|^p e^{-(\pi/2)p(a x^2+ a^{-1} \xi^2 - 2 |x| |\xi|)} \; dx d\xi < \infty.
\end{eqnarray*}
It remains to show that this is only possible for $R \equiv 0$. We are linked to prove that  
\[ \int_0^\infty \int_0^\infty  |R(u,v)|^p \; e^{-|u-v|^2} dudv=\infty, \] 
for any non-zero polynomial $R$. But non-vanishing polynomials are bounded below, say for $|u|>A, \; |v|>A$, then
\[\int_A^\infty \int_A^\infty  e^{-|u-v|^2} dudv=\infty.  \] 
This completes the proof.
\end{proof}
%But polynomials that are not identically zero are bounded below
\begin{remark}
Theorem \ref{Th2} implies Theorem \ref{Th1}. To see this,
assume that the condition (\ref{eq01}) of Theorem \ref{Th1} is satisfied. Then
\begin{eqnarray*}
\iint_{\mathbb{R}^{2d}} |V_g f(x,\xi)|^p \; e^{\pi p |x \cdot \xi|} \;dxd\xi 
\leq \iint_{\mathbb{R}^{2d}} |V_g f(x,\xi)|^p \; e^{\pi p (x^2+\xi^2)/2} \;dxd\xi < \infty.
\end{eqnarray*}
By Theorem \ref{Th2}, $f \equiv 0$ or $g \equiv 0$. 
\end{remark}

\begin{proof}[Proof of Theorem \ref{Th3}]
Let $1 \leq p,q <\infty$ with $1/p+1/q=1$. We choose $N>0$, such that 
\begin{equation}\label{eq8}
\iint_{\mathbb{R}^{2d}} \frac{1}{(1+|x|+|\xi|)^{Nq}} \; dx d\xi < \infty.
\end{equation}
Then using H\"older's inequality, we get
\begin{eqnarray}\label{eq9}
&& \iint_{\mathbb{R}^{2d}} \frac{|f(x) \hat{f}(\xi)|}{(1+|x|+|\xi|)^{N}} \; e^{2 \pi |x \cdot \xi|} \; dx d\xi  \nonumber \\
&& \leq  \left( \iint_{\mathbb{R}^{2d}} |f(x) \hat{f}(\xi)|^p \; e^{2 \pi p |x \cdot \xi|} \; dx d\xi \right)^{\frac{1}{p}} \left( \iint_{\mathbb{R}^{2d}} \frac{1}{(1+|x|+|\xi|)^{Nq}} \; dx d\xi \right)^{\frac{1}{q}}. 
\end{eqnarray}
Thus the assumptions (\ref{eq03}) and (\ref{eq8}) imply that 
\[\iint_{\mathbb{R}^{2d}} \frac{|f(x) \hat{f}(\xi)|}{(1+|x|+|\xi|)^{N}} \; e^{2 \pi |x \cdot \xi|} \; dx d\xi < \infty.  \]
Hence, the condition (\ref{Bon-eq}) of Theorem \ref{Bonami} is satisfied for $f$ and we conclude that $f=0$ almost everywhere whenever $N \leq d$ and if $N >d$, then $f(x)= P(x)e^{-a \pi x^2}$ where $P$ is a polynomial of degree $< \frac{N-d}{2}$ and $a>0$. Indeed, we show that if $f$ is of this form, then $f \equiv 0$.

Let $f(x)= P(x)e^{-a \pi x^2}$, then $\hat{f}(\xi)=Q(\xi)e^{-\pi a^{-1} \xi^2}$, for some polynomial $Q$. Since $f$ satisfies (\ref{eq03}), we have 
\begin{eqnarray*}
\iint_{\mathbb{R}^{2d}} |f(x) \hat{f}(\xi)|^p \; e^{2 \pi p |x \cdot \xi|} \; dx d\xi
= \iint_{\mathbb{R}^{2d}} |P(x)|^p |Q(\xi)|^p \; e^{- \pi p(ax^2+a^{-1}\xi^2 -2 |x| |\xi|)} \; dx d\xi < \infty.
\end{eqnarray*}
It remains to show that this is only possible for $P \equiv 0$. We are linked to prove that
\[ \int_0^\infty \int_0^\infty  |P(u)|^p\; |Q(v)|^p \; e^{-|u-v|^2} dudv=\infty, \] 
for any non-zero polynomials $P$ and $Q$. But non-vanishing polynomials are bounded below, say for $|u|>A, \; |v|>A$, then
\[\int_A^\infty \int_A^\infty  e^{-|u-v|^2} dudv=\infty.  \] 
This completes the proof.
\end{proof}
%But polynomials that are not identically zero are bounded below

\begin{remark}
\begin{enumerate}
\item[$(i)$] If we consider $p=1$ in Theorems \ref{Th1} and \ref{Th2}, then we obtain the Theorem \ref{Gro1} and Conjecture \ref{Con}, respectively. 
\item[$(ii)$] If we consider $p=1$ in Theorem \ref{Th3}, then we obtain the Theorem \ref{Beurling}.
\end{enumerate}
\end{remark}

Next, we discuss an analogue of Donoho--Stark uncertainty
principle and provide some estimates for the size of the essential support of $V_g f$. We start with the following lemma.
\begin{lemma}\label{lem4}
Let $1 \leq p < \infty$ and $f,g \in L^2(\R^d)$. If $U \subseteq \R^{2d}$ and $\epsilon \geq 0$ are such that
\[\iint_U |V_g f(x,\xi)|^p \;dx d\xi \geq (1-\epsilon) \|f \|^p_2 \|g \|^p_2, \]
then $|U| \geq 1-\epsilon$.
\end{lemma}
\begin{proof}
The Cauchy--Schwartz inequality implies that
\[ |V_g f(x,\xi)|=|\langle f, M_\xi T_x g \rangle| \leq \|f \|_2 \|g \|_2 \qquad \text{for all\;} (x,\xi) \in \R^{2d}. \]
Therefore,
\[ (1-\epsilon) \|f \|^p_2 \|g \|^p_2 \leq \iint_U |V_g f(x,\xi)|^p dx d\xi \leq \|V_g f \|^p_\infty |U| \leq |U| \|f \|^p_2 \|g \|^p_2,   \]
and so $|U| \geq 1-\epsilon$.
\end{proof}
Estimates obtained in Theorems \ref{Th4} and \ref{Th5} improve Lemma \ref{lem4} and provide a stronger estimate on the size of the essential support of $V_g f$. To prove Theorems \ref{Th4} and \ref{Th5} we use Lieb's \cite{Lie90} uncertainty principle. 

\begin{theorem}\label{Lie-ineq}
$($Lieb \cite{Lie90}$)$ Assume that $f,g \in L^2(\R^d)$. Then
\begin{equation}\label{eq10}
\iint_{\R^{2d}} |V_g f(x,\xi)|^p dx d\xi  \quad \left\{
\begin{array}{ll}
\leq \left(\frac{2}{p} \right)^d (\|f \|_2 \; \|g \|_2)^p  & \quad \mbox{if } 2 \leq p < \infty, \\
\geq \left(\frac{2}{p} \right)^d (\|f \|_2 \; \|g \|_2)^p & \quad \mbox{if } 1 \leq p \leq 2 .
\end{array}
\right.
\end{equation}
\end{theorem}

\begin{proof}[Proof of Theorem \ref{Th4}]
We first apply H\"older's inequality with exponents $q = p$ and
$q'= \frac{p}{p-1} $, and then in the second step we use Lieb's inequality for $p \geq 2$, we obtain
\begin{eqnarray*}
 (1-\epsilon) &&  \|f \|_2 \|g \|_2 \leq  \iint_U |V_g f(x,\xi)| \;dx d\xi \\
&& \leq \left(\iint_{\R^{2d}} |V_g f(x,\xi)|^p \;dx d\xi \right)^{\frac{1}{p}} \left(\iint_{\R^{2d}} \chi_U (x,\xi)^{q'} \;dx d\xi \right)^{\frac{p-1}{p}} \\
&& \leq \left(\frac{2}{p} \right)^{\frac{d}{p}} \|f \|_2 \; \|g \|_2 \; |U|^{\frac{p-1}{p}}.
\end{eqnarray*}
Thus  
\[ |U| \geq (1-\epsilon)^{\frac{p}{p-1}} \left( \frac{p}{2}  \right)^{\frac{d}{p-1}} \qquad \text{for all } \; p \geq 2. \]
\end{proof}

\begin{proof}[Proof of Theorem \ref{Th5}]
Using Lieb's inequality for $p = 1$ and (\ref{eq002}) we obtain that
\[ (1-\epsilon) \|V_g f \|^p_1 \geq (1-\epsilon) 2^{pd}\; \|f \|^p_2 \; \|g \|^p_2 =(1-\epsilon) 2^{pd}\; \|V_g f \|^p_2. \]
On the other hand, using H\"older's inequality with exponents $q = \frac{2}{p}$ and $q'= \frac{2}{2-p} $, for $1 \leq p < 2$ we get
\begin{eqnarray*}
\iint_U |V_g f(x,\xi)|^p dx d\xi 
& \leq & \left(\iint_{\R^{2d}} |V_g f(x,\xi)|^2 \;dx d\xi \right)^{\frac{p}{2}} \left(\iint_{\R^{2d}} \chi_U (x,\xi)^{q'} \;dx d\xi \right)^{\frac{2-p}{2}} \\
& = & \|V_g f \|^p_2 \; |U|^{\frac{2-p}{2}}.
\end{eqnarray*}
Combining these inequalities with (\ref{eq003}), we obtain
\[ (1-\epsilon) 2^{pd}\; \|V_g f \|^p_2 \leq |U|^{\frac{2-p}{2}} \; \|V_g f \|^p_2 . \]
Thus
\[ |U| \geq 2^{\frac{2pd}{2-p}}  (1-\epsilon)^{\frac{2}{2-p}} \qquad \text{for all  } \; 1 \leq p < 2. \]
\end{proof}

\begin{remark}
If we consider $p = 1$ in Theorem \ref{Th5}, then we get the estimate obtained by Gr\"ochenig (\cite{Gro03}, Prop. 2.5.2.). Thus Theorem \ref{Th5} generalizes the estimate of Gr\"ochenig.  Also, if we compare Theorems \ref{Th4} and \ref{Th5}, then we see that Theorem \ref{Th5} gives a slightly sharper estimate.
\end{remark}

\section*{Acknowledgments}
The author is deeply indebted to Prof. S. Thangavelu for several fruitful discussions and generous comments. The author also wishes to thank Prof. Aline Bonami for several valuable comments and suggestions concerning Theorems \ref{Th2} and \ref{Th3}. Further, the author is grateful to the University Grants Commission, India for providing the Dr. D. S. Kothari Post Doctoral Fellowship (Award No.- F.4-2/2006 (BSR)/MA/18-19/0032).


\begin{thebibliography}{10}


\bibitem{Bon03}
A. Bonami, B. Demange and P. Jaming, {\it Hermite functions and uncertainty principles for the Fourier and the windowed Fourier transforms}. Rev. Mat. Iberoamericana 19:23--55 (2003).

\bibitem{Cow83}
M.G. Cowling and J.F. Price, {\it Generalizations of Heisenberg's inequality}. In: Harmonic Analysis (G. Mauceri, F. Ricci and G. Weiss (eds.)), LNM 992, pp. 443--449, Springer, Berlin (1983).

\bibitem{Don89}
D.L. Donoho and P.B. Stark, {\it Uncertainty principles and signal recovery}. SIAM J. Appl. Math. 49(3):906--931 (1989).

\bibitem{Dem05}
B. Demange, {\it Uncertainty principles for the ambiguity function}. J. London Math. Soc. (2) 72(3):717--730 (2005).

\bibitem{Gro2001}
K. Gr\"ochenig, {\it Foundations of Time--Frequency Analysis}, Birkh\"auser, Boston (2001).

\bibitem{Gro03}
K. Gr\"ochenig, {\it Uncertainty principles for time--frequency representations}. In: Advances in Gabor analysis (H. G. Feichtinger and T. Strohmer (eds.)), pp. 11--30, Birkh\"auser, Boston (2003).

\bibitem{Gro01}
K. Gr\"ochenig and G. Zimmermann, {\it Hardy's theorem and the short-time Fourier transform of Schwartz functions}. J. London Math. Soc. (2) 63(1):205--214 (2001).

\bibitem{Har33}
G.H. Hardy, {\it A theorem concerning Fourier transforms}. J. London Math. Soc. 8:227--231 (1933).

\bibitem{hav94}
V. Havin and B. J\"oricke, {\it The uncertainty principle in harmonic analysis}. In: A Series of Modern Surveys in Mathematics. 28, Springer--Verlag, Berlin (1994).

\bibitem{Hor91}
L. H\"ormander, {\it A uniqueness theorem of Beurling for Fourier transform pairs}. Ark. Mat. 29:237--240 (1991).

\bibitem{Jam98}
P. Jaming, {\it Principe d'incertitude qualitatif et reconstruction de phase pour la transform{\'e}e de Wigner}. C. R. Acad. Sci. Paris S{\'e}r. I Math. 327(3):249--254 (1998).

\bibitem{Lie90}
E.H. Lieb, {\it Integral bounds for radar ambiguity functions and Wigner distributions}. J. Math. Phys. 31(3):594--599 (1990).

\bibitem{Mal10}
E. Malinnikova, {\it Orthonormal sequences in $L^2(\R^d)$ and time frequency localization}. J. Fourier Anal. Appl. 16(6):983--1006 (2010).

\bibitem{Thn04}
S. Thangavelu, {\it An Introduction to the Uncertainty Principle}, Progr. Math. 217, Birkh\"auser, Basel (2004).




\end{thebibliography}
\end{document}